\documentclass[oneside,12pt]{amsart}
\usepackage[utf8]{inputenc}
\usepackage{amssymb,graphicx,hyperref}
\usepackage[scale=0.7]{geometry}

\newcommand{\R}{\mathbb R}
\newcommand{\norm}[1]{\left\lVert #1 \right\rVert}
\newcommand{\diagmon}{misaligned}
\DeclareMathOperator{\conv}{conv}
\DeclareMathOperator{\inte}{int}

\theoremstyle{definition}
\newtheorem{thm}{Theorem}
\newtheorem{lem}[thm]{Lemma}
\newtheorem{prop}[thm]{Proposition}

\newtheorem{obs}[thm]{Observation}

\theoremstyle{remark}
\newtheorem*{remark}{Remark}
\theoremstyle{definition}
\newtheorem{defi}[thm]{Definition}

\title{Total orders realizable as the distances between two sets of points}

\author[G.L. Maldonado]{Gerardo L. Maldonado}
\address[G.L. Maldonado]{Centro de Ciencias Matemáticas, UNAM Campus Morelia, Morelia, Mexico}
\email{gmaldonado@matmor.unam.mx}

\author[E. Roldán-Pensado]{Edgardo Roldán-Pensado}
\address[E. Roldán-Pensado]{Centro de Ciencias Matemáticas, UNAM Campus Morelia, Morelia, Mexico}
\email{e.roldan@im.unam.mx}

\author[M. Raggi]{Miguel Raggi}
\address[M. Raggi]{Escuela Nacional de Estudios Superiores, UNAM Campus Morelia, Morelia, Mexico}
\email{mraggi@enesmorelia.unam.mx}

\keywords{Euclidean distances; Total orders; Convex sets}

\begin{document}

\begin{abstract}
	In this note we give a negative answer to a question proposed by Almendra-Hernández and Martínez-Sandoval. Let $n\le m$ be positive integers and let $X$ and $Y$ be sets of sizes $n$ and $m$ in $\R^{n-1}$ such that every pair of points in $X\cup Y$ defines a unique distance. There is a natural order on $X\times Y$ induced by the distances between the corresponding points. The question is if all possible orders on $X\times Y$ can be obtained in this way. We show that the answer is  negative when $n<m$. The case $n=m$ remains open.
\end{abstract}

\maketitle

\section{Introduction}
A common question in discrete geometry is, given a point set $S$, how many different distances can there be between points of $S$? This is known as the \emph{Erdős distance problem} \cite{GIS2011}. Another related question is how to know when a graph is a unit distance graph, \textit{i.e.}, when it can be realized geometrically in $\R^d$ with edges of unit length \cite{ALON2014}. In the plane, this decision problem is NP-hard; specifically, it is complete for the existential theory of the reals \cite{SCHAEFER2013}.
In this paper we are not focused on exact distances, only on how the edge-lengths are ordered.

Let $n\ge 2$ and $d$ be positive integers and let $P=\{p_0,\dots,p_{n-1}\}$ be a set of points in $\R^d$ in general position (\textit{i.e.} no $d+1$ elements from $P$ lie on the same hyperplane). Assume that the points of $P$ are such that every pair of points in $P$ defines a unique distance. Then $P$ induces a natural order on $\binom{[n]}{2}$ (where $[n]=\{0,\dots,n-1\}$ and $\binom{[n]}{2}$ denotes the ordered pairs $(i,j)\in [n]\times[n]$ with $i<j$), given by
\begin{equation*}
	(i,j) < (k,l) \iff \norm{p_i-p_j} < \norm{p_k-p_l}.
\end{equation*}
We say an order on $\binom{[n]}{2}$ is \emph{realizable in $\R^d$} if there exists a set of points $P\subset\R^d$ in general position which induces this order.

A natural question becomes: Given $d$ and $n$, which orders on $\binom{[n]}{2}$ are realizable in $\R^d$? This is always possible when $n$ is small compared to $d$.	

In \cite{AM22}, the authors study the problem of finding, for a given $n$, the smallest $d$ such that every order on $\binom{[n]}{2}$ can be obtained from some $P\subset \R^d$. They claim that the answer is $d=n-2$.
While they do prove that any order can be realized in $\R^{n-2}$, the order which they claim is unrealizable in $\R^{n-3}$ is in fact realizable.
We show this in Section \ref{sec:coro4}, however this leaves open the question of whether $d= n-2$.

Now assume that $2\le n\le m$ and $d$ are positive integers. Consider sets of points $P=\{p_0,\dots,p_{n-1}\}$ and $Q=\{q_0,\dots,q_{m-1}\}$ in $\R^d$ such that $P$ and $Q$ are in general position and every pair of points in $P\times Q$ determines a unique distance. Then $P$ and $Q$ induce a natural order on $[n]\times [m]$ given by
\begin{equation*}
	(i,j) < (k,l) \iff \norm{p_i-q_j} < \norm{p_k-q_l}.
\end{equation*}
As before, we say an order on $[n]\times [m]$ is \emph{realizable in $\R^d$} if there exist sets of points $P,Q\subset\R^d$ which induce this order.

In this setting, Almendra-Hernández and Martínez-Sandoval proved that the minimum dimension $d$ for which every order on $[n]\times[m]$ is realizable in $\R^d$ satisfies $n-1\le d\le n$. 

Our main contribution is showing that the correct value of $d$ is precisely $n$. We do this by constructing an order on $[n]\times [n+1]$ that is not realizable in $\R^{n-1}$. This statement is made precise in Theorem~\ref{thm:example} below.

The question that remains is whether or not every order on $[n]\times [n]$ can be obtained by a pair of sets in $\R^{n-1}$.
An exhaustive computer analysis shows this to be true when $n\le 3$, but we have not been able to prove it for larger values of $n$. 

This paper is organized as follows. In Section \ref{sec:problem} we state our main results together with the necessary definitions.
In Section \ref{sec:proofs} we give the proof of our main theorem. In Section \ref{sec:coro4} we explain the gap in the proof of one of the theorems from \cite{AM22}. Finally, in Section \ref{sec:compu} we present the results of some computer some experiments related to Theorem \ref{thm:example}.

\section{Definitions and main result statement}\label{sec:problem}

We start by defining a specific type of order on $[d+1]\times[d+2]$. Later we show that these orders cannot be realized in $\R^d$.

\begin{defi}\label{def:order}
	Let $d$ be a positive integer and let $<$ be a total order on $[d+1]\times[d+2]$. We say that order $<$ is \emph{\diagmon} if the following properties are satisfied: 
	\begin{enumerate}
		\item[1.] $(k,k) < (k,k-1)< \dots < (k,k-d)$, for all $k\in [d+1]$,
		\item[2.] $(k,k) < (k+1,k) < \dots < (k+d,k)$, for all $k\in [d+1]$, and
		\item[3.] $(d,d+1) < (d-1,d+1) < \dots < (0,d+1)$,
	\end{enumerate}
	where the entries of the first and second coordinates of the elements of $[d+1]\times[d+2]$ are taken modulo $d+1$ and $d+2$, respectively.
\end{defi}

A useful way of understanding {\diagmon} orders is by using the bijection
\begin{equation*}
	\varphi:[d+1]\times[d+2]\to[(d+1)(d+2)]
\end{equation*}
such that
\begin{equation*}
	(i,j) < (k,l) \iff \varphi((i,j)) < \varphi((k,l)).
\end{equation*}
We may then visualize the order $<$ with a table as in Figure~\ref{fig:table}.

The properties from Definition \ref{def:order} can be interpreted as follows. Property 1 states that in every row, the element on the diagonal is the smallest and increases as we move to the right. After reaching the end of the table, we continue from the leftmost element until we meet the diagonal again. Property 2 is similar but for the first $d+1$ columns. For the last column, Property 3 implies that the elements decrease from top to bottom. Figure \ref{fig:table} shows a visual representation of this.

\begin{figure}
	\begin{center}
		\includegraphics{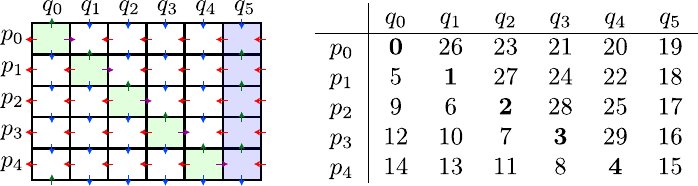}
	\end{center}
	\caption{Left: The arrows represent order restrictions of Definition \ref{def:order}, (as a torus). Right: An example for $d=4$ of a {\diagmon} order.}
	\label{fig:table}
\end{figure}

It is easy to construct examples of {\diagmon} orders: Place the first $d+1$ numbers in the main diagonal of a $(d+1)\times(d+2)$ table (in any order). Next fill in the diagonals below the main diagonal one by one so that each diagonal contains the smallest numbers available. Once that is done, fill in the last column with the next $d+1$ in increasing order from bottom to top. Finally fill in the remaining spaces, one diagonal at a time starting from the top and always using the smallest numbers available for each diagonal. The resulting table satisfies the properties in Definition \ref{def:order}.

\begin{remark}
	Note that this construction produces
	\begin{equation*}
		(d+1)!\left(d!(d-1)!\cdots 1!\right)^2
	\end{equation*}
	{\diagmon} orders. However, this is still quite small compared to the total number of possible orders.
\end{remark}

\begin{thm}\label{thm:example}
	No {\diagmon} order on $[d+1]\times[d+2]$ is realizable in $\R^d$.
\end{thm}

\section{Proof of Theorem~\ref{thm:example}}\label{sec:proofs}

First, let us introduce some standard notation. If $X$ is a finite subset of $\R^d$, denote by $\conv(X)$ the convex hull of $X$. If $|X|=d+1$, denote by $O(X)$ the circumcenter of $X$.

Assume, for the sake of contradiction, that $P=\{p_0,\dots,p_d\}$ and $Q=\{q_0,\dots,q_{d+1}\}$ are subsets of $\R^d$ which induce a {\diagmon} order on $[d+1]\times [d+2]$.

For each $i\in[d+1]$, let $Q_i = Q\setminus \{q_i\}$. Our goal is to arrive at a contradiction by showing that $O(P)$ must be both inside and outside the simplex $\conv(Q_{d+1})$.

We begin by noting some properties that sets $P$ and $Q$ must satisfy as a consequence of inducing a {\diagmon} ordering.
\begin{obs}\label{obs:dm}
	From Definition \ref{def:order} we deduce the following:
	\begin{enumerate}
		\item\label{obs1} Every point in $P$ is closer to $q_0$ than to $q_{d+1}$.
		\item\label{obs2} If $i\in[d]$, then $q_{i+1}$ and $q_{d+1}$ are closer to $p_{i+1}$ than to $p_i$, but every other point in $Q$ is closer to $p_i$.
	\end{enumerate}
\end{obs} 

Let $M(x,y)$ denote the perpendicular bisector (hyperplane) of segment $\overline{xy}$. Then $M(x,y)$ separates the points that are closer to $x$ than to $y$ from the points that are closer to $y$ than to $x$. In this way, Observation \ref{obs:dm} can be visualized in Figure \ref{fig:obs}. We denote the closed halfspace bounded by $M(x,y)$ which contains $y$ by $M^+(x,y)$.

\begin{figure}
	\begin{center} 
		\includegraphics{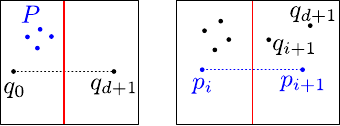} 
	\end{center} 
	\caption{Visualization of Observation~\ref{obs:dm}. The red solid lines are the perpendicular bisectors of the blue dashed segments. Points in $P$ are blue while points from $Q$ are black.}
	\label{fig:obs}
\end{figure} 

The next lemma gives us important properties about the circumcenters of $Q_{d+1}$, $Q_0$ and $P$. This is accomplished using some relevant facts about cones (for more about convex cones, see \cite{GUNTER96} and \cite{BMS97}).

\begin{lem}\label{lem:circum}
	Let $X=\{x_0,\dots,x_d\}$ and $Y=\{y_0,\dots,y_d\}$ be subsets of $\R^d$ in general position such that every pair of points in $X\times Y$ determines a unique distance. Suppose that for every $i$
	\begin{equation*}
		\norm{y_i-x_i}<\norm{y_i-x_{i+1}}<\dots<\norm{y_i-x_{i-1}},
	\end{equation*}
	where the indices are taken mod $d+1$. Then $O(X)$ is in the interior of $\conv(Y)$.
\end{lem}
\begin{proof}
	For simplicity during this proof, all indices are taken mod $d+1$. Also, assume that $O(X)$ is the origin so that the hyperplanes $M(x_i,x_j)$ pass through the origin.
	
	The hyperplanes $M(x_i,x_j)$ split $\R^d$ into several cones which may be identified in the following way. If $y\in \R^d$ is a point which does not lie in any of the hyperplanes $M(x_i,x_j)$, then we may assign to $y$ the permutation $\sigma = (n_0,n_1,\dots n_d)$ that represents the order in which the distances from $y$ to the points $x_{n_i}$ appear. In other words, $\norm{y-x_{n_0}}<\dots<\norm{y-x_{n_d}}$. Every $y$ in a given cone is assigned the same permutation, and points in distinct cones are assigned distinct permutations, so we may identify each cone by the permutation assigned to any of its points.
	
	By hypothesis, every pair of points in $X\times Y$ determines a unique distance, so we have that each $y_i$ is in the interior of a cone. Let $C_i$ be the cone that contains $y_i$. Then the permutation assigned to $C_i$ is $(i,i+1,\dots,i-1)$. Note that
	\begin{equation*}
		C_i=M^+(x_{i+1},x_{i})\cap M^+(x_{i+2},x_{i+1})\cap\dots\cap M^+(x_{i-1},x_{i-2}).
	\end{equation*}
	
	Recall that if $C$ is a cone, its dual cone is defined as $C^* = \{x : \langle x,y\rangle \geq 0, y\in C\}$.
	Let $v_j=x_{j}-x_{j+1}$ for each $j\in[d+1]$.
	
	Thus, $C_i^*$ is the positive cone generated by the set of vectors $\{x_{i}-x_{i+1},x_{i+1}-x_{i+2},\dots,x_{i-2}-x_{i-1}\}=\{v_0,\dots,v_d\}\setminus\{v_{i-1}\}$ (see Figure~\ref{fig:cono}). 
	
	\begin{figure} 
		\begin{center}   
			\includegraphics{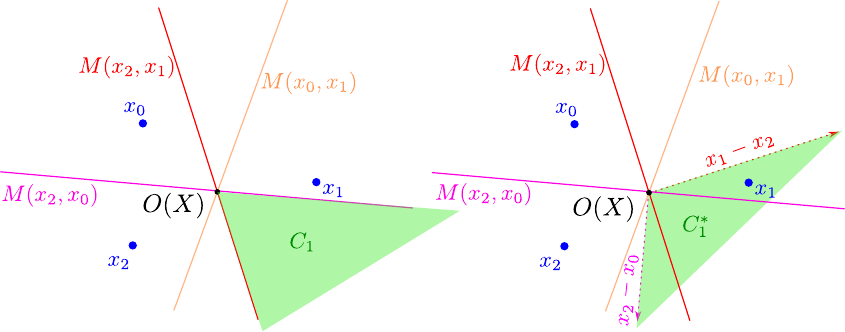}
			\caption{Left: Cone $C_1$, corresponding to the permutation $\sigma = (1,2,0)$, shown. Right: Cone $C_1^*$ shown.}\label{fig:cono}
		\end{center} 
	\end{figure}
	
	The cones $C_i^*$ cover $\R^d$. To see why this is true, observe that, since $X$ is in general position, every subset with $d$ elements from the set $\{v_0,\dots,v_d\}$ is a basis of $\R^d$. Note that $\sum_{i\in[d+1]}v_i=0$ and take $w\in\R^d$, then $w = \sum_{i\in[d]}\lambda_i v_i$ with $\lambda_i\in\R$, for $i\in[d]$. By taking $\lambda>0$ large enough so that $\lambda_i+\lambda>0$ for each $i\in[d]$, we may write $w = \sum_{i\in[d]}(\lambda_i+\lambda) v_i+\lambda v_d$. This shows that $w$ can be written as a linear combination of the $v_i$ with positive coefficients and is therefore in some $C_i^*$.
	
	Assume that $O(X)$ is not in the interior of $\conv(Y)$. Then there is a non-zero vector $v$ such that $\langle v,y_i\rangle \ge 0$ for every $i\in[d+1]$, so $\langle -v,y_i\rangle \le 0$ for every $i\in[d+1]$.
	On the other hand, since each $y_i$ lies in the interior of $C_i$, we have $\langle y_i,c\rangle >0$ for all $c\in C_i^*$. Furthermore, because the dual cones $C_i^*$ cover $\R^d$, there exists some $i\in[d+1]$ for which $-v\in C_i^*$. For this particular $i$, we obtain $\langle y_i,-v\rangle >0$, a contradiction.
\end{proof}

We apply Lemma \ref{lem:circum} three times:
\begin{itemize}
	\item With $x_i=p_i$ and $y_i=q_i$ for $i\in[d+1]$. Here, the hypothesis is satisfied because of Property 2 of Definition \ref{def:order}.
\end{itemize}
For the following two applications, consider Lemma \ref{lem:circum} changing the indices in the hypothesis to:
\begin{equation*}
	\norm{y_i-x_i}<\norm{y_i-x_{i-1}}<\dots<\norm{y_i-x_{i+1}},
\end{equation*}
where the indices are taken mod $d+1$.
\begin{itemize}
	\item With $x_i=q_{i}$ and $y_i=p_i$ for $i\in[d+1]$ (where the indices in $x$ are taken mod $d+1$). Here, the hypothesis is satisfied because of Property 1 of Definition \ref{def:order}.
	\item With $x_0 = q_{d+1}$ and $x_i=q_{i}$, for $i\in[d+1]\setminus{0}$. Also, take $y_i=p_i$ for $i\in[d+1]$. Here, the hypothesis is satisfied because of Property 1 of Definition \ref{def:order}.
\end{itemize}

In this way we obtain that:
\begin{align*}
	O(P)&\in\conv(Q_{d+1}),\\
	O(Q_{d+1})&\in\conv(P)\text{ and}\\
	O(Q_0)&\in\conv(P). 
\end{align*}

Now let $Q' = Q\setminus\{q_0,q_{d+1}\}$ and let $H$ denote the hyperplane spanned by $Q'$. Next we prove that $q_0$ and $q_{d+1}$ must lie in the same open halfspace bounded by $H$.

Consider the circumsphere $S_0$ of $Q_0$. Let $B_{d+1}$ the ball with center $Q_0$ bounded by $S_0$. Notice that $q_{d+1} \in S_0$. 

Analogously, let $S_{d+1}$ be the circumsphere of $Q_{d+1}$ and let $B_{d+1}$ be the ball with center $O(Q_{d+1})$ bounded by $S_{d+1}$. Notice that $q_0 \in S_{d+1}$.

Since both $O(Q_{0})$ and $O(Q_{d+1})$ lie in $\conv(P)$, part (\ref{obs1}) of Observation \ref{obs:dm} implies that both $O(Q_0)$ and $O(Q_{d+1})$ are closer to $q_0$ than to $q_{d+1}$. 

Therefore $q_0 \in B_{0}$ and $q_{d+1} \notin B_{d+1}$. This means that
\begin{align}
	q_{0} & \in S_{d+1}\cap B_{0},\nonumber\\
	q_{d+1} & \in S_{0}\setminus B_{d+1}\label{eq:notin}
\end{align}
and consequently $q_0$ and $q_{d+1}$ are on the same side of $H$ (see Figure~\ref{fig:intersectingspheres}).

\begin{figure}
	\begin{center}
		\includegraphics{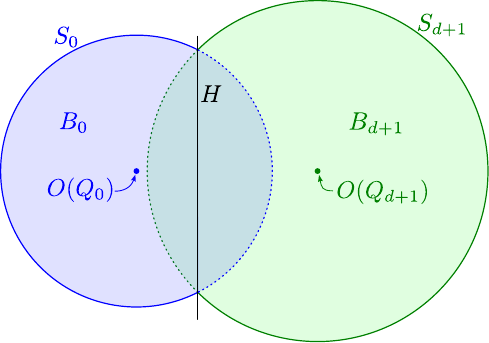}
		\caption{$q_0$ must be in the dashed green line and $q_{d+1}$ must be in the solid blue line, so $q_0$ and $q_{d+1}$ lie on the same side of $H$ as $O(Q_0)$.}
		\label{fig:intersectingspheres}
	\end{center}
\end{figure}

Let $H^+$ be the closed halfspace bounded by $H$ which contains $q_0$ and $q_{d+1}$.

Part (\ref{obs2}) of Observation~\ref{obs:dm} implies that $M(p_i,p_{i+1})$ separates $\{q_{i+1},q_{d+1}\}$ from the rest of $Q$ for each $i\in [d]$ (see Figure \ref{fig:obs}). In order to exploit this we use the following lemma which is a corrected version of Corollary 4 from \cite{AM22} (see Section \ref{sec:coro4}).

\begin{lem}\label{lem:coro4}
	Let $X=\{x_0,\dots, x_d\}\subset\R^d$ be a set of points in general position, $O$ be a point in the interior of $\conv(X)$ and $L_0,\dots,L_d\subset\R^d$ be hyperplanes in general position. For each $i\in[d+1]$, define $L_i^+$ as a closed halfspace bounded by $L_i$ that contains $O$ and define $L_i^-$ as the closure of its complement. Assume that $O\not\in L_0$ and
	\begin{equation*}
		x_i\in \left(\bigcap_{j\neq i} L_j^-\right)\cap L_i^+.
	\end{equation*}
	Then
	\begin{equation*}
		O\in\bigcap_{i\in[d+1]}L_i^+\subset\conv(X).
	\end{equation*}
\end{lem}
\begin{figure}
	\begin{center}   
		\includegraphics{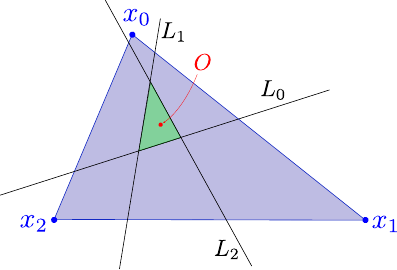}
		\caption{Example of the Lemma \ref{lem:coro4} in the plane. The green intersection must lie inside the blue simplex.}\label{fig:lem5}   
	\end{center}
\end{figure}
This Lemma is exemplified in Figure~\ref{fig:lem5}.
\begin{proof}
	The statement is trivial for $d=1$, so assume $d>1$. Let $L:=\bigcap_{i\in[d+1]} L_i^+$. Assume to the contrary that there exists a point $y$ in $L$ with $y\notin \conv(X)$. Note that we may assume $y\in \inte(L)$, since $L$ is a non-empty convex set, the hyperplanes $L_i$ are in general position (so $\inte(L)$ is not empty), and $\conv(X)$ is closed.
	
	As $y\not\in\conv(X)$ but $O \in \conv(X)$, there exists a point $z\in \inte(L)$ in the segment between $O$ and $y$ which lies in the boundary of $\conv(X)$. Therefore there is some $j\in[d+1]$ such that $z$ is a convex combination of the $x_i$ with $i\neq j$.
	
	Since $x_i\in L_j^-$ for $i\neq j$, we conclude that $z\in L_j^-$. But the segment between $O$ and $y$ is contained in $L_j^+$. The only way it can have a point, besides $O$ and $y$, in $L_j^-$ is if $O,y\in L_j$ which contradicts our choice of $y$.
\end{proof}

As consequence of Lemma \ref{lem:circum}, $O(P)$ is contained in the interior of $\conv(Q_{d+1})$ and therefore $O(P)\not\in H$. Then, we can apply Lemma \ref{lem:coro4} with $O=O(P)$, $x_i=q_i$, $L_0=H$ and $L_i=M(p_{i-1},p_i)$ to conclude that $q_{d+1}\in\conv(Q_{d+1})\subset B_{d+1}$, which contradicts \eqref{eq:notin}.
\qed

\section{On the upper bound from \cite{AM22}}\label{sec:coro4}

Corollary 4 from \cite{AM22} is false as stated. The induction argument they use fails because the base hypothesis no longer holds after the first step.
Consequently, their Proposition 5 is also false. They use this proposition in order to assert that not all orders on $\binom{[d+3]}{2}$ are realizable in $\R^d$.
We show that the orders they give are in fact realizable in $\R^d$ for every $d\ge 4$ and so the question of finding if every order on $\binom{[d+3]}{2}$ are realizable in $\R^d$ remains open.

\begin{prop}[Counterexample to Proposition 5 from \cite{AM22}]
	Let $d\ge 4$. Then there is a set $P=\{p_0,\dots,p_{d+2}\}$ of points in $\R^d$ in general position such that each pair of distinct points in $P$ defines a unique distance and the order that $P$ induces on $\binom{[d+3]}{2}$ satisfies the following conditions:
	\begin{enumerate}
		\item for any $(i, j)\in \binom{[d]}{2}$ and $(k, l) \in \binom{[d+3]}{2} \setminus \binom{[d]}{2}$ we have $(k, l) < (i, j)$, and
		\item for any $(i, j)\in [d]\times \{d,d+1,d+2\}$ and $(k, l) \in \binom{[d+3]}{2}\setminus([d]\times \{d,d+1,d+2\})$ we have $(i, j) < (k, l)$.
	\end{enumerate}
\end{prop}
These conditions simply mean that the largest elements are the ones in $\binom{[d]}{2}$ and the smallest are the ones in $[d]\times \{d,d+1,d+2\}$.
\begin{proof}
	For $d=4$ we have the following example which was found via a computer search, as detailed in Section \ref{sec:compu}.
	\begin{align*}
		P&=\{(16,39,7,51),(26,14,0,9),(48,39,47,49),(68,17,27,10),\\
		&\qquad (22,13,37,23),(40,54,18,17),(55,18,8,45)\}.
	\end{align*}
	For the general case $d\ge 5$, consider the space $\R^d$ as the set of points in $\R^{d+1}$ whose coordinates add up to $1$.
	
	Let $X=\{x_0,\dots,x_{d}\}$ be the canonical basis in $\R^{d+1}$ so that $\norm{x_i-x_j}=\sqrt{2}$ whenever $i<j$.
	Let $c=(x_0+\dots+x_{d})/(d+1)$ and, for $i\in[3]$, let
	\begin{equation}
		y_i = c+\sqrt{2}\left(\frac{x_{2i}+x_{2i+1}}{2}-c\right)
	\end{equation}
	so that the vectors in $Y=\{y_0,y_1,y_2\}$ form an equilateral triangle of side-length
	\begin{equation}
		\norm{y_1-y_0} = \sqrt{2}\frac{\norm{x_2+x_3-x_0-x_1}}{2}=\sqrt{2}.
	\end{equation}
	Thanks to the symmetry of our vectors, the norm of the difference between $x\in X$ and $y\in Y$ is either $\norm{y_0-x_0}$ or $\norm{y_0-x_2}$. A simple calculation shows that $\norm{y_0-x_0}^2=2-\sqrt 2-\frac{3-2\sqrt 2}{d+1}<2-\sqrt{2}\approx 0.586$ and $\norm{y_0-x_2}^2 = 2-\frac{3-2\sqrt 2}{d+1}<2$.
	In any case, we have that $\norm{y-x}<\sqrt{2}$.
	
	Note that all the points in $X$ and $Y$ have coordinates that add up to $1$ and the set $X\cup Y$ has one more point than needed. So, the desired point set $P$ can be obtained by a suitable perturbation of the terminal points of the vectors in $(X\setminus\{x_d\})\cup Y$. To make this more precise, for a suitably small $\varepsilon>0$ (depending on $d$) we may choose numbers $\varepsilon_0,\dots,\varepsilon_{d+2}$ uniformly at random in $(0,\varepsilon)$ use the vectors $\{(1+\varepsilon_i) x_i:i\in[d]\}$ and $\{(1-\varepsilon_{d+i})y_i:i\in[3]\}$. In this way the three points in $Y$ are moved closer together while the points in $X$ are moved apart while maintaining the required conditions.
\end{proof}

\section{Computational experiments}\label{sec:compu}

Given $d,n,m$ we aim to construct examples of sets of points in $\mathbb{R}^d$ corresponding to given orders on $[n]\times [m]$ by formulating the construction as an optimization problem.

For a fixed permutation $p$ of $[n]\times [m]$ do the following.
\begin{enumerate}
	\item Start with two random sets of points in $\mathbb{R}^d$, $A$ and $B$ of sizes $n$ and $m$ respectively.
	\item Compute all pairwise distances between $A$ and $B$. Then, construct a permutation $p'$ associated with the order of these pairwise distances.
	\item Consider the error as follows: 
	$$E(A,B,p) = L_1(p,p') + D(A,B) + C(A,B),$$
	where $L_1$ is the Manhattan distance between the permutations $p = (n_0,n_1,\dots n_d)$ and $p' = (n_0',n_1',\dots n_d')$ defined by $L_1(p,p')=\sum_{i=0}^{d}\lvert n_i-n_i'\rvert$, $D$ penalizes cases where pairwise distances become too similar, and $C$ penalizes instances where two points in $A\cup B$ are too close to each other. Both penalties help prevent optimization stagnation and mitigate numerical errors by avoiding ambiguity from nearly identical numerical values.
	\item Optimize this function and stop when it reaches $0$, as this means an example has been found.
\end{enumerate}
We used differential evolution with parameters chosen heuristically, and ran thousands of simulations on a GPU (Nvidia RTX 2080 Ti). The implementation we used can be found on \href{https://github.com/mraggi/DistancesInOrder}{https://github.com/mraggi/DistancesInOrder}. 

Here is a summary of some results that were obtained by using the above method:
\begin{itemize}
	\item For $d=2$, $n=3$ and $m=3$, we were able to show that all $9!$ permutations are realizable by explicitly constructing a point set that realizes each permutation.
	\item For $d=2$, $n=3$, and $m=4$, out of $10\,000$ random permutations, we found $11$ for which we were not able to produce a proper set of points. In contrast, when sampling permutations uniformly at random, the expected number of {\diagmon} permutations is $1/15\,120$.
	\item For $d=3$, $n=4$, and $m=4$, we managed to produce a set of points that realizes each of the $10\,000$ random permutations considered.
	\item For $d=3$, $n=4$, and $m=5$, out of $10\,000$ random permutations, there were $26$ permutations cases for which we did not manage to produce a set of points which realizes them.
\end{itemize}

Given this evidence, we conjecture that all permutations of $[d+1]\times[d+1]$ can be realized in $\R^d$.

\section*{Acknowledgments}

This work was supported by UNAM-PAPIIT IN111923.

\bibliographystyle{amsalpha}
\bibliography{refs}

\end{document}